\documentclass[10pt]{amsart}

\usepackage{amsmath}
\usepackage{amsthm}
\usepackage{amssymb}
\usepackage[all]{xy}
\usepackage{color}
\usepackage{graphicx}
\usepackage{todonotes}

\theoremstyle{plain}
\newtheorem{lemma}[subsection]{Lemma}
\newtheorem{proposition}[subsection]{Proposition}
\newtheorem{theorem}[subsection]{Theorem}
\newtheorem{corollary}[subsection]{Corollary}
\theoremstyle{definition}

\newtheorem{example}[subsection]{Example}
\newtheorem{remark}[subsection]{Remark}
\newtheorem*{theorem*}{Theorem}

\DeclareMathOperator*{\hocolim}{hocolim}

\newcommand{\co}{\colon\thinspace}

\newcommand{\ra}{\rightarrow}
\newcommand{\inv}{^{-1}}
\newcommand{\id}{\mathrm{id}}

\newcommand{\Cat}{\mathit{Cat}}% The category of small categories 
\newcommand{\cB}{\mathcal B}
\newcommand{\cC}{\mathcal C}
\newcommand{\cD}{\mathcal D}

\newcommand{\cL}{\mathcal L}
\newcommand{\cS}{\mathcal S}

\newcommand{\br}{\mathsf{Br}^{\!+}}

\title{Weak braided monoidal categories and their homotopy colimits}
\date{\today}

\author{Mirjam Solberg} \address{
    Department of Mathematics, University of Bergen, P.O. Box 7800, N-5020 Bergen, Norway} \email{mirjam.solberg@math.uib.no}

\begin{document}
\begin{abstract}
 We show that the homotopy colimit construction for diagrams of categories with an operad action, recently introduced by Fiedorowicz, Stelzer and Vogt, has the desired homotopy type for diagrams of weak braided monoidal categories. This provides a more flexible way to realize $E_2$ spaces categorically. 
\end{abstract}

\maketitle

\section{Introduction}
Braided monoidal categories have been much studied and are used extensively in many areas of mathematics, for instance in knot theory, representation theory and topological quantum field theories. It has been known for a long time that the nerve of a braided monoidal category is an $E_2$ space, and it was shown recently \cite{FSV13} that all homotopy types of $E_2$ spaces arise  in this way. In this article we study a weaker categorical structure, namely weak braided monoidal categories. These are monoidal categories with a family of natural morphisms $X\otimes Y\ra Y\otimes X$ satisfying the axioms for a braiding, except that they are not required to be isomorphisms. We will see that weak braided monoidal categories give a more flexible way to realize $E_2$ spaces categorically.

Homotopy colimit constructions have become increasingly important in homotopy theory. In order for the equivalence between weak braided monoidal categories and $E_2$ spaces to be really useful one should be able to construct homotopy colimits on the categorical level. Such a homotopy colimit construction was defined in \cite{FSV13} in general for diagrams of categories with an operad action. The question of the homotopy properties of the homotopy colimit was left open for weak braided monoidal categories. In this paper we provide an answer to that question. Let $\br\text{-}\Cat$ denote the category of weak braided monoidal categories and let $X$ be a diagram of weak braided monoidal categories. Applying the nerve $N$ to a weak braided monoidal category yields a space with an action of the $E_2$ operad $N\br$, see Subsection~\ref{subsection hom}. Let ${\hocolim}^{\br}X$ denote the homotopy colimit of $X$ defined in \cite{FSV13}, and let  ${\hocolim}^{N\br}NX$ denote the homotopy colimit of $NX$, for details see Subsection~\ref{subsection hom}. Then our main result, Theorem~\ref{theorem main}, can be stated as follows.
\begin{theorem*}
There is a natural weak equivalence 
$${\hocolim}^{N\br}NX \ra N({\hocolim}^{\br}X)$$
of $N\br$-algebras.
\end{theorem*}

\subsection{Organization}
We begin by giving the definition of weak braided monoidal categories in Section~\ref{section wbm} and provide some examples. In Section~\ref{section hocolim} we set up and prove our main result, Theorem~\ref{theorem main}. The proof involves an analysis of braid monoids which is interesting in its own right.

\section{Weak braided monoidal categories}\label{section wbm}
Let $\cD$ be a monoidal category with monoidal product $\otimes$, monoidal unit $i$, associativity isomorphisms $\mathsf{a}$ and left and right unit isomorphisms $\mathsf{l}$ and $\mathsf{r}$ respectively. 
A \emph{weak braiding} for $\cD$ consists of a family of morphisms 
$$\mathsf{b}_{d,e}\co d\otimes e \ra e\otimes d$$
in $\cD,$ natural in $d$ and $e$, such that $\mathsf{l}_d\mathsf{b}_{d,i}=\mathsf{r}_d$, and 
$\mathsf{r}_d\mathsf{b}_{i,d}=\mathsf{l}_d$ and the following two diagrams 
\begin{equation*}
\xymatrix{
(e\otimes d)\otimes f \ar[r]^{{\mathsf{a}}} 
& e\otimes (d\otimes f) \ar[d]+<0.7cm,0.3cm>^(0.7){\mathrm{id}\otimes{\mathsf{b}}} 
&& d\otimes (f\otimes e) \ar[r]^{{\mathsf{a}}^{-1}} 
& (d\otimes f)\otimes e \ar[d]+<0.7cm,0.3cm>^(0.7){{\mathsf{b}}\otimes1} \\
\save[]+<-0.7cm,0cm>*{(d\otimes e)\otimes f}
\ar[]+<-0.65cm,0.3cm>;[u]^(0.3){{\mathsf{b}}\otimes\mathrm{id}} \ar@[]+<0cm,-0.3cm>;[d]_(0.3){{\mathsf{a}}}  \restore 
&\save[]+<0.7cm,0cm>*{e\otimes (f\otimes d)} \restore  
&&\save[]+<-0.7cm,0cm>*{d\otimes (e\otimes f)}
\ar[]+<-0.65cm,0.3cm>;[u]^(0.3){\mathrm{id}\otimes{\mathsf{b}}} \ar[]+<-0.65cm,-0.3cm>;[d]_(0.3){{\mathsf{a}}^{-1}}  \restore 
&\save[]+<0.7cm,0cm>*{(f\otimes d)\otimes e} \restore  \\
d\otimes (e\otimes f) \ar[r]_{{\mathsf{b}}} 
& (e\otimes f)\otimes d \ar[u]+<0.7cm,-0.3cm>_(0.7){{\mathsf{a}}}
&& (d\otimes e)\otimes f \ar[r]_{{\mathsf{b}}} 
& f\otimes (d\otimes e) \ar[u]+<0.7cm,-0.3cm>_(0.7){{\mathsf{a}}^{-1}}
}\end{equation*}
commute for all $d,e$ and $f$ in $\cD$.
Here the sub indices of the weak braiding $\mathsf{b}$ and the associativity isomorphism $\mathsf{a}$ have been omitted. 
A \emph{weak braided monoidal category} is a monoidal category equipped with a weak braiding. 
Note that if all the morphisms $\mathsf{b}_{d,e}$ are isomorphisms, then $\mathsf{b}$ is a braiding for the monoidal category.

\begin{remark}
 The notion of a weak braided monoidal category found in \cite{BFSV03} and \cite{FSV13} differs from the definition given here, in that the underlying monoidal structure is required to be strictly associative and strictly unital. This is not a significant difference since each weak braided monoidal category is equivalent to a weak braided strict monoidal category, along monoidal functors preserving the weak braiding. The proof of this is similar to the proof of the analogous result for braided monoidal structures.
 \end{remark}

Weak braided monoidal categories have not been much studied in the literature, so before we proceed we will look at some examples to show how such structures naturally arise. The first example, the disjoint union of the braid monoids, is somehow the canonical example.

\begin{example}
Let $\cB_m^+$ denote the braid monoid on $m$ strings with the following presentation:$$\langle\sigma_1,\ldots,\sigma_{m-1}\,|\,\sigma_i\sigma_j=\sigma_j\sigma_i \text{ if } |i-j|>1 \text{ and } 
\sigma_i\sigma_{i+1}\sigma_i =\sigma_{i+1}\sigma_i\sigma_{i+1}\rangle.$$
The elements in $\cB_m^+$ are called positive braids on $m$ strings, or just positive braids. 

Let $\cB^+$ denote the category with one object $\mathbf{m}$ for each integer $m\geq 0$, with endomorphisms of $\mathbf{m}$ the braid monoid $\cB_m^+$ and no other morphisms. This is a strict monoidal category with $\mathbf{m}\otimes\mathbf{n}= \mathbf{m\!+\!n}$, and $\mathbf{0}$ as a unit. The weak braiding from $\mathbf{m}\otimes\mathbf{n}$ to $\mathbf{n}\otimes\mathbf{m}$ is given by the positive braid 
$$(\sigma_n\cdots\sigma_{m+n-1})\cdots(\sigma_2\cdots\sigma_{m+1})(\sigma_1\cdots\sigma_{m}),$$
braiding the first $m$ strings over the last $n$ strings. This is the same as the usual braiding in the classical braid category, which is the disjoint union of the braid groups, see \cite[Example~2.1]{JS93}.
\end{example}

\begin{example}
We consider the category of  non-negatively graded abelian groups. An object $G$ is a collection of abelian groups $G_n$ for $n\geq0$. A morphism $f\co G\ra H$ consists of group homomorphisms $f_n\co G_n\ra H_n$ for $n\geq 0$. This category has a monoidal product given by 
$$(G\otimes H)_n= \bigoplus_{n_1+n_2=n} G_{n_1}\otimes H_{n_2}.$$
Now fix an integer $k$. For $g\in G_{n_1}$ and $h\in H_{n_2}$ the assignment $g\otimes h\mapsto k^{n_1 n_2} h\otimes g$ induces a map from $G_{n_1}\otimes H_{n_2}$ to $H_{n_2}\otimes G_{n_1}$, which in turn induces a homomorphism $(G\otimes H)_n \ra (H\otimes G)_n$. The collection of such maps gives a weak braiding for the category of  non-negatively graded abelian groups. 
Note that if $k$ is a unit, i.e. $\pm 1$, then the weak braiding is an actual braiding. 

This example may be generalized to the category of non-negatively graded $R$-modules for any commutative ring $R$. Pick an element in $R$ to play the role of $k$ in the weak braiding.
\end{example}

A much studied construction is the center of a monoidal category, which can be endowed with a braided monoidal structure, see for instance Example~2.3 in \cite{JS93}. Our next example is a weak version of this. 

\begin{example}
 Let $\cD$ be a strict monoidal category with monoidal unit $i$. We consider pairs $(d,\delta)$ where $d$ is an object in $\cD$ and $\delta$ is a natural transformation $\delta\co d\otimes (-) \ra (-)\otimes d$ such that $\delta_i= \id_d$ and such that for any two objects $x,y\in\cD$ the triangle
 $$
 \xymatrix{
 d\otimes x\otimes y \ar[rr]^{\delta_x\otimes \id_y} \ar[dr]_{\delta_{x\otimes y}} 
 && x \otimes d\otimes y \ar[dl]^{\id_x\otimes\delta_y } \\
 &x\otimes y \otimes d
 }$$
 commutes.
 An arrow between two pairs $(d,\delta)\ra (e,\epsilon)$ consists of a morphism $\phi\co d\ra e$ such that for all $x\in\cD$ the identity  $\epsilon_x\circ(\phi\otimes \id_x)=(\id_x\otimes \phi)\circ \delta_x$ holds. We can define a monoidal product of two such pairs by setting 
 $$(d, \delta)\otimes (e,\epsilon)=(d\otimes e, (\delta\otimes\id_e)\circ(\id_d\otimes\epsilon)).$$
 The collection of morphisms 
 $$\delta_e\co (d, \delta)\otimes (e,\epsilon) \ra  (e,\epsilon)\otimes(d, \delta)$$
 satisfies the conditions for a weak braiding on this category of pairs and arrows. We call this the weak center of $\cD$.
 
 The requirement that $\cD$ should be strictly associative and strictly unital was only a matter of convenience. A similar construction works for any monoidal category, details are left to the interested reader.
\end{example}

\subsection{Operadic interpretation of weak braided monoidal structures} 
When the underlying monoidal multiplication is strict, weak braided monoidal categories are the algebras over a certain $\Cat$-operad.
By a $\Cat$-operad we understand an operad internal to the category $\Cat$ of small categories.
Following \cite[Section~8]{FSV13} we introduce the $\Cat$-operad $\br$ such that $\br$-algebras are weak braided strict monoidal categories. The objects of $\br(k)$ are the elements $A\in \Sigma_k$.
Let $p\co \cB_k^+\ra\Sigma_k$ denote the projection of the braid monoid onto the corresponding symmetric group. Then a morphism $\alpha\colon A\to B$ in $\br(k)$ is a positive braid $\alpha\in \cB_k^+$ such that 
$p(\alpha) A=B$.
 Composition in $\br(k)$ is given by multiplication in $\cB_k^+$. The category $\br(k)$ has a right action of the symmetric group on $k$ letters defined on objects and morphisms by sending $\alpha\colon A\to B$ to $\alpha\colon Ag\to Bg$ for $g\in\Sigma_k$. The operad structure map 
$$\gamma\colon \br(k)\times \br(j_1)\times\dots\times\br(j_k)\to \br(j_1+\dots+j_k)$$
takes the tuple $(A,B_1,\dots,B_k)$ to
$$A(j_1,\dots,j_k)\circ (B_1\sqcup\dots\sqcup B_k).$$
Here $A(j_1,\dots,j_k)$ denotes the canonical block permutation obtained from $A$ by replacing the $i$th letter with $j_i$ letters. The action on morphisms is analogous except for the obvious permutation of the indices. 
It is easy to check that the category of weak braided monoidal categories with weak braiding preserving strict monoidal functors is isomorphic to $\br$-algebras. See for instance the argument given in Section~5.1 in \cite{SS14} for the braided monoidal version.
We denote the category of $\br$-algebras by $\br$-$\Cat$. 

\section{Homotopy colimits of weak braided monoidal categories}\label{section hocolim}

In \cite[Definition~4.10]{FSV13} there is a general homotopy colimit construction for a diagram of algebras over a $\Cat$-operad. Let $\cL$ be a small category and consider 
the category $(\br\text{-}\Cat)^\cL$ of functors $\cL\ra \br\text{-}\Cat$ and natural transformations. The above mentioned construction gives in particular a functor 
$${\hocolim_\cL}^{\br}\co (\br\text{-}\Cat)^\cL \ra \br\text{-}\Cat.$$

\subsection{The homotopy type of the homotopy colimit}\label{subsection hom}
Let $\cS$ be the category of simplicial sets and let $N$ be the nerve functor from $\Cat$ to $\cS$.   If we apply $N$ levelwise to the $\Cat$-operad $\br$, we get an operad $N\br$ internal to the category $\cS$. We denote the category of algebras over $N\br$ as $N\br\text{-}\cS$. A morphism of $N\br$-algebras is called a weak equivalence if the underlying simplicial set map is a weak equivalence. 
For a diagram $W\co\cL\ra (N\br\text{-}\cS)$, let $\hocolim_\cL^{N\br}W $ denote the coend construction $N(-/\cL)\otimes_\cL QW$, where $Q$ is an object wise cofibrant replacement functor. This is the homotopy colimit of $QW$ from Definition~18.1.2 in \cite{MR1944041}. 
If $X$ is in $(\br\text{-}\Cat)^\cL$, then there is a natural map 
$${\hocolim_\cL}^{N\br}NX \ra N({\hocolim_\cL}^{\br}X),$$
see the paragraph before Definition~6.7 \cite{FSV13}. This is an operadic version of Thomason's map in Lemma~1.2.1 \cite{Tho79}.
The question if this map is a weak equivalence or not, was left open  in \cite{FSV13}.
Our main result provides a positive answer to this problem.

\begin{theorem}\label{theorem main}
The diagram 
 $$\xymatrix{(\br\text{-}\Cat)^\cL \ar[r]^N \ar[d]_{\hocolim_\cL^{\br}}
 & (N\br\text{-}\cS)^\cL \ar[d]^{\hocolim_\cL^{N\br}} \\
 \br\text{-}Cat \ar[r]^N & N\br\text{-}\cS
 }$$
 commutes up to weak equivalence of $N\br$-algebras.
\end{theorem}

The operad $N\br$ is an $E_2$-operad, see Proposition~8.13 in \cite{FSV13}. The above theorem gives one way to relate weak braided monoidal categories and $E_2$-spaces as seen in the corollary below. Fiedorowicz, Stelzer and Vogt obtain the same equivalence without using the homotopy colimit construction of $\br$-algebras in \cite{FSV}. 

\begin{corollary}\label{corollary}
We have an equivalence of localized categories 
$$(\br\text{-}\Cat)[we\inv]\simeq (N\br\text{-}\cS)[we\inv].$$
\end{corollary}

\begin{proof}
 Theorem~\ref{theorem main} shows that Theorem~7.6 in \cite{FSV13} applies to the operad  $\br$. The corollary then follows from the latter theorem with the added observation that the localization $(\br\text{-}\Cat)[we\inv]$ exists, see Proposition~A.1 in \cite{SS14}.
\end{proof}

By general theory (details will be provided later), the proof of the theorem reduces to showing that certain categories have the property that each connected component has an initial object. 
Fix an $A\in\Sigma_m$, a $B\in \Sigma_n$, and non-negative integers $r_1, \ldots, r_n$ such that $r_1+ \cdots +r_n= m$. Let $\tilde{B}$ denote the canonical block permutation $B(r_1, \ldots, r_n)\in\Sigma_m$ obtained from $B$ by replacing the $i$th letter with $r_i$ letters. We define a poset category $\cC$ depending on $A$, $B$ and $r_1, \ldots, r_n$ . The objects in $\cC$ are the positive braids $\alpha\in \cB_m^+$ such that 
 $$p(\alpha)A\tilde{B} \in(\Sigma_{r_1}\times\cdots\times\Sigma_{r_n}) \subseteq\Sigma_m.$$
There is a morphism $\alpha \leq \beta$ from $\alpha$ to $\beta$ in $\cC$ if there exit $\gamma_i\in\cB_{r_1}^+$ for $i=1,\ldots,n$ such that 
 $(\gamma_1\oplus\cdots\oplus\gamma_n)\alpha=\beta$ in $\cB_m^+$.

\subsection{Analysis of minimal positive braids in $\cC$}
We call an object $\nu$ in $\cC$ a \emph{minimal object} if for all objects $\nu'$ in $\cC$, $\nu'\leq \nu$ implies $\nu' = \nu$. Define the norm  $|\beta|\in \mathbb N_0$ of an element $\beta$ in $\cB_m^+$, as the length of any word representing $\beta$. This is well defined because the relations in 
the presentation of $ \cB_m^+$ only involve equations with the same number of generators on each side. Details on this word problem can be found in  Sections~6.1.3 and 6.5.1 in \cite{KT08}. Then it is immediate from the definition that $\nu$ is a minimal object 
if and only if $\nu\neq(\gamma_1\oplus\cdots\oplus\gamma_n)\nu'$ for all
$\gamma_i\in B_{r_i}^+$ and all $\nu'\in \cC$ with $|\nu'|<|\nu|$.

\begin{proposition}\label{prop unique morphism}
Given an object $\alpha$ in $\cC$ there is a unique minimal object $\nu_\alpha$ such that $\nu_\alpha\leq \alpha$.
\end{proposition}

Before we prove this result we will review the concept of least common multiples in a monoid. A right common multiple of two elements $\gamma$ and $\gamma'$ in $\cB_k^+$  is an element in 
$\cB_k^+$ that is of the form 
$\gamma\phi=\gamma'\phi'$ for some $\phi$ and $\phi'$ in $\cB_k^+$. 
A right least common multiple of $\gamma$ and $\gamma'$ is an element 
$\mathrm{lcm}_k(\gamma,\gamma')\in \cB_k^+$ such that $\mathrm{lcm}_k(\gamma,\gamma')$ 
is a right common multiple of $\gamma$ and $\gamma'$, and such that any right common multiple 
of $\gamma$ and $\gamma'$ is of the form $\mathrm{lcm}_k(\gamma,\gamma')\omega$ 
for some $\omega\in \cB_k^+$.
A unique right least common multiple of $\gamma$ and $\gamma'$ exists for any two elements 
$\gamma$ and $\gamma'$ in $\cB_k^+$, see Theorem 6.5.4 in \cite{KT08}. Since we will only be dealing with right least common multiples, and not left least common multiples, the notation $\mathrm{lcm}$ will not be ambiguous. 

\begin{lemma}
Given $\gamma_i,\gamma_i'\in\cB_{r_i}^+$ for $i=1,\ldots,n$, let $\gamma=\gamma_1\oplus\cdots\oplus\gamma_n$ and similarly $\gamma'=\gamma_1'\oplus\cdots\oplus\gamma_n'$. 
Then the least common multiple $\mathrm{lcm}_r(\gamma, \gamma')$ in $\cB_{r}^+$ of $\gamma$ and $\gamma'$ is equal to
$$
\mathrm{lcm}_{r_1}(\gamma_1,\gamma_1')\oplus\cdots\oplus\mathrm{lcm}_{r_n}(\gamma_n,\gamma_n').$$
\end{lemma}
\begin{proof}
 It is clear that $\mathrm{lcm}_{r_1}(\gamma_1,\gamma_1')\oplus\cdots\oplus\mathrm{lcm}_{r_n}(\gamma_n,\gamma_n')$ is a right multiple of both $\gamma$ and $\gamma'$. Therefore 
 $$\mathrm{lcm}_r(\gamma, \gamma')\phi= 
\mathrm{lcm}_{r_1}(\gamma_1,\gamma_1')\oplus\cdots\oplus\mathrm{lcm}_{r_n}(\gamma_n,\gamma_n')$$
for some $\phi\in \cB_{m}^+$. We claim that 
$\mathrm{lcm}_r(\gamma, \gamma')$ lies in $\cB_{r_1}^+\times\cdots\times\cB_{r_n}^+$ as well. From this it is straightforward to check that $\mathrm{lcm}_{r_1}(\gamma_1,\gamma_1')\oplus\cdots\oplus\mathrm{lcm}_{r_n}(\gamma_n,\gamma_n')$ is the right least common multiple of $\gamma$ and $\gamma'$ in the monoid $\cB_{r_1}^+\times\cdots\times\cB_{r_n}^+$ and the result follows.

 The claim made above follows from a more general observation that if we have two positive braids $\alpha, \beta\in\cB_r^+$  such that their
product $\beta\alpha$ lies in  $\cB_{r_1}^+\times\cdots\times\cB_{r_n}^+$, then both $\alpha$ and $\beta$ have to lie in $\cB_{r_1}^+\times\cdots\times\cB_{r_n}^+$ as well. This is a consequence of the nature of the presentation of $\cB_{r}^+$ given earlier. For each of the relations in the presentation, the two sides of the equation involves exactly the same generators. This implies that if a generator is present in one word representing a positive braid, it will be present in all words representing that positive braid. For details on the word problem in braid monoids, see Section~6.1.5 in \cite{KT08}.
\end{proof}

\begin{proof}[Proof of Proposition \ref{prop unique morphism}]
We first prove the existence of $\nu_\alpha$. If $\alpha$ is not a minimal object, there exists an object $\alpha_1\in \cC$ such that $\alpha_1\leq\alpha$ and $|\alpha_1|<|\alpha|$.
We repeat this process as many times as necessary until we obtain a minimal $\alpha_k$ with $\alpha_k\leq\alpha_{k-1}$ and $|\alpha_k|<|\alpha_{k-1}|$. The process terminates after a finite number of steps since the norm of the $\alpha_i$'s decrease strictly each time. We set $\nu_\alpha=\alpha_k$, so by construction $\nu_\alpha\leq\alpha$.

We now turn to the uniqueness of $\nu_\alpha$.
 Suppose there are two minimal objects $\nu_\alpha$ and $\nu_\alpha'$ such that both $\nu_\alpha\leq \alpha$ and $\nu_\alpha'\leq \alpha$. Then $\alpha$ equals both $(\gamma_1\oplus\cdots\oplus\gamma_n)\nu_\alpha$ and  $(\gamma_1'\oplus\cdots\oplus\gamma_n')\nu_\alpha'$ for some $\gamma_i,\gamma_i'\in\cB_{r_i}^+$, $i=1,\ldots,n$.  
 Abbreviating $\gamma_1\oplus\cdots\oplus\gamma_n$ to $\gamma$
and $\gamma_1'\oplus\cdots\oplus\gamma_n'$ to $\gamma'$, we recall that 
$$\mathrm{lcm}_r(\gamma, \gamma') = 
\mathrm{lcm}_{r_1}(\gamma_1,\gamma_1')\oplus\cdots\oplus\mathrm{lcm}_{r_n}(\gamma_n,\gamma_n').$$
Since $\alpha$ is a right common multiple of  both $\gamma$ and $\gamma'$, 
$$\alpha=(\mathrm{lcm}_{r_1}(\gamma_1,\gamma_1'))\oplus\cdots\oplus(\mathrm{lcm}_{r_n}(\gamma_n,\gamma_n'))\omega $$
for some $\omega\in\cB_m^+$.
The right least common multiple of $\gamma_i$ and $\gamma_i'$ is in particular a 
right common multiple of $\gamma_i$ and $\gamma_i'$, so
$ \mathrm{lcm}(\gamma_i,\gamma_i')=\gamma_i\phi_i=\gamma_i'\phi_i'$
for some $\phi_i,\phi_i'\in\cB_{r_i}^+$, $i=1,\ldots,n$. Combining this we get 
that 
\begin{align*}
 (\gamma_1\oplus\cdots\oplus\gamma_n)\nu_\alpha=\alpha&=(\gamma_1\oplus\cdots\oplus\gamma_n)(\phi_1\oplus\cdots\oplus\phi_n)\omega \quad\text{and} \\
 (\gamma_1'\oplus\cdots\oplus\gamma_n')\nu_\alpha'=\alpha&=(\gamma_1'\oplus\cdots\oplus\gamma_n')(\phi_1'\oplus\cdots\oplus\phi_n')\omega.
\end{align*}
The braid monoid injects into the corresponding braid group \cite[Theorem~6.5.4]{FSV13}, so we can apply left cancellation to the above equations to obtain 
\begin{align*}
 \nu_\alpha=(\phi_1\oplus\cdots\oplus\phi_n)\omega \quad\text{ and } \quad
 \nu_\alpha'=(\phi_1'\oplus\cdots\oplus\phi_n')\omega.
\end{align*}
It is straightforward to check that $p(\omega)A\tilde{B}$ is in 
$\Sigma_{r_1}\times\cdots\times\Sigma_{r_n}$, so that $\omega$ is an object in $\cC$. Then the above equations say that $\omega\leq \nu_\alpha$ and $\omega\leq \nu_\alpha'$ in $\cC$. But since $\nu_\alpha$ and $\nu_\alpha'$ are minimal objects these maps have 
to be identities. This proves the uniqueness of $\nu_\alpha$. 
\end{proof}

\begin{lemma}\label{lemma initial object}
 Each connected component in $\cC$ has an initial object.
\end{lemma}

\begin{proof}
 Given a morphism $\alpha\leq\beta$ in $\cC$, let $\nu_\alpha$ and $\nu_\beta$ denote the two minimal objects associated to $\alpha$ and $\beta$ respectively, by the previous result. The two objects must be equal since $\nu_\alpha\leq \alpha \leq \beta$, but $\nu_\beta$ is the unique minimal object with $\nu_\beta\leq \beta$.   Hence the minimal objects associated to any two objects in the same connected component has to be equal, and we have a unique minimal object in each connected component of $\cC$. The minimal objects are initial in their respective connected components.
\end{proof}

Fix an $M\in\Sigma_m$, an $N\in \Sigma_n$, and non-negative integers $s_1, \ldots, s_n$ such that $s_1+ \cdots +s_n= m$. Let $\tilde{N}$ denote the canonical block permutation $N(s_1, \ldots, s_n)\in\Sigma_m$ obtained from $N$ by replacing the $i$th letter with $s_i$ letters.
The \emph{factorization category} $\mathcal C (M,N,s_1,\ldots,s_n)$, as defined in \cite[Section 6]{FSV13}, has as objects tuples $(C_1,\ldots,C_n, \alpha)$ consisting of $C_i\in\Sigma_{s_i}$ for $i=1,\ldots n$, and 
$\alpha \in \cB_m^+$ such that 
\begin{equation}\label{eq}
 p(\alpha)M=\tilde{N}(C_1\oplus\cdots\oplus C_n).
\end{equation}
  A morphism  from 
$(C_1,\ldots,C_n, \alpha)$ to $(D_1,\ldots,D_n, \beta)$ consists of elements 
$\gamma_i$ in 
$\cB_{s_i}^+$ for $i=1,\ldots n$ such that $(\gamma_1\oplus\cdots\oplus\gamma_n)\alpha=\beta$.

\begin{lemma}\label{lemma factorization category iso}
 The factorization category $\cC(A,B\inv,r_{B\inv(1)},\ldots,r_{B\inv(n)})$ is isomorphic to the category $\cC$ considered in this section.
\end{lemma}

\begin{proof} Here 
$\tilde{B}\inv(C_1\oplus\cdots\oplus C_n)=(C_{B(1)}\oplus\cdots\oplus C_{B(n)})\tilde{B}\inv$, 
so Equation \eqref{eq} can be rewritten as $$p(\alpha)A\tilde{B}=C_{B(1)}\oplus\cdots\oplus C_{B(n)}.$$ 
This equation determines the $C_i$'s uniquely given $\alpha$ with $p(\alpha)A\tilde{B}$ in \mbox{$\Sigma_{r_1}\times\cdots\times\Sigma_{r_n}$}.
The two categories therefore have isomorphic objects, and the morphism sets are easily seen to be isomorphic as well. 
\end{proof}

\begin{proof}[Proof of Theorem~\ref{theorem main}]
Together Lemmas \ref{lemma initial object} and \ref{lemma factorization category iso} show that each factorization category has an initial object in each of its connected components. Thus the operad $\br$ satisfies the factorization condition \cite[Definition~6.8]{FSV13} and the result follows from Theorem~6.10 \cite{FSV13}.\end{proof}

\bibliographystyle{alpha}
%\bibliography{Bibliography}

\end{document}